\documentclass{amsart}

\usepackage{amssymb}
\usepackage{amsmath}
\usepackage{amsthm}
\usepackage{graphicx}
\usepackage{bm}

\newtheorem{thm}{Theorem}[section]
\newtheorem{cor}[thm]{Corollary}

\newtheorem{pro}[thm]{Proposition}
\theoremstyle{definition}

\theoremstyle{remark}
\newtheorem{rem}[thm]{Remark}
\numberwithin{equation}{section}

\newcommand{\set}[1]{\left\{#1\right\}}

\newcommand{\CC}{\mathbb{C}}
\newcommand{\ZZ}{\mathbb{Z}}
\newcommand{\E}{\mathbb{E}}
\newcommand{\PPP}{\mathbb{P}}

\newcommand{\X}{\mathbf{X}}
\newcommand{\Y}{\mathbf{Y}}

\newcommand{\V}{\mathbf{V}}

\newcommand{\Z}{\mathbf{Z}}

\newcommand{\0}{\mathbf{0}}

\newcommand{\uu}{\mathbf{u}}

\newcommand{\vv}{\mathbf{v}}

\newcommand{\UU}{\bm{U}}

\newcommand{\WW}{\bm{W}}
\newcommand{\EE}{\bm{E}}

\newcommand{\DD}{\bm{D}}

\newcommand{\A}{\bm{A}}

\newcommand{\C}{\bm{C}}
\newcommand{\I}{\bm{I}}

\newcommand{\M}{\bm{M}}

\newcommand{\T}{\bm{T}}
\newcommand{\OO}{\bm{O}}
\newcommand{\f}{\bm{f}}

\newcommand{\etav}{\bm{\eta}}

\newcommand{\xiv}{\bm{\xi}}
\newcommand{\Phiv}{\bm{\Phi}}
\newcommand{\Siga}{\bm{\Sigma}}

\newcommand{\Lam}{\bm{\Lambda}}
\newcommand{\av}{\bm{a}}
\newcommand{\bv}{\bm{b}}

\newcommand{\wv}{\bm{w}}
\newcommand{\phiv}{\bm{\phi}}

\newcommand{\deltav}{\bm{\delta}}
\newcommand{\Deltav}{\bm{\Delta}}
\newcommand{\muv}{\bm{\mu}}

\newcommand{\F}{\bm{F}}

\newcommand{\Lada}{\bm{\Lambda}}

\newcommand{\psiv}{\bm{\psi}}

\newcommand{\proj}{\mathrm{Proj}}
\newcommand{\diag}{\mathrm{diag}}
\newcommand{\rk}{\mathrm{rank}}
\newcommand{\cov}{\mathrm{Cov}}

\newcommand{\tr}{\, \mathrm{tr}}

\newcommand{\spanot}{\overline{\mathrm{span}}}
\newcommand{\spant}{\mathrm{Span}}

\begin{document}

\title{Regular multidimensional stationary time series}%
\author{Tam\'as Szabados}%
\address{Budapest University of Technology and Economics}%
\email{szabados@math.bme.hu}%


\begin{abstract}
The aim of this paper is to give a simpler, more usable sufficient condition to the regularity of generic weakly stationary time series. Also, this condition is used to show how regular processes satisfying these sufficient conditions can be approximated by a lower rank \emph{regular} process. The relevance of these issues is shown by the ever increasing presence of high-dimensional data in many fields lately, and because of this, low rank processes and low rank approximations are becoming more important. Moreover, regular processes are the ones which are completely influenced by random innovations, so they are primary targets both in the theory and applications.
\end{abstract}

\maketitle


\section{Introduction}\label{se:intro}

Let $\X_t = (X_t^1, \dots , X_t^d)$, $t \in \ZZ$, be a $d$-dimensional weakly stationary time series, where each $X_t^j$ is a complex valued random variable on the same probability space $(\Omega, \mathcal{F}, \PPP)$. It is a second order, and in this sense, translation invariant process:
\[
\E \X_t = \muv \in \CC^d, \qquad \E((\X_{t+h}-\muv)(\X^*_t - \muv^*)) = \C(h) \in \CC^{d \times d}, \qquad \forall t,h \in \ZZ,
\]
where $\C(h)$, $h \in \ZZ$, is the non-negative definite covariance matrix function of the process. Without loss of generality, from now on it is assumed that $\muv = \0$.

Thus the considered random variables will be $d$-dimensional, square integrable, zero expectation random complex vectors in the Hilbert space $L^2_d(\Omega, \mathcal{F}, \PPP)$, with inner product
\[
\langle \X, \Y \rangle := \sum_{j=1}^d \E(X^j \overline{Y^j}) = \tr \; \cov(\X, \Y) .
\]
However, the \emph{orthogonality }of the random vectors $\X$ and $\Y$ is defined by the stronger relationship $\X \perp \Y \Leftrightarrow \cov(\X, \Y) = \E(\X \Y^*) = \OO$. If
\begin{align*}
M &= \spant\{X^{\ell}_{\gamma} : \ell = 1, \dots , d; \gamma \in \Gamma\} := \left\{\sum_{j=1}^n
\sum_{\ell =1}^d a_{j\ell}  X^{\ell }_{\gamma_j} : a_{j\ell}  \in \CC,
 n \ge 1\right\} ,  \\
M_d &= \spant_d\{\X_{\gamma} : \gamma \in \Gamma\} := \left\{\sum_{j=1}^n \A_j \X_{\gamma_j} : \A_j \in \CC^{d \times d}, n \ge 1\right\} ,
\end{align*}
then $M_d = M \times \cdots \times M$, a Cartesian product with $d$ factors, where $M \in L^2(\Omega, \mathcal{F}, \PPP)$ and $M_d \in L^2_d(\Omega, \mathcal{F}, \PPP)$.

The \emph{past of $\{\X_t\}$ until time $n \in \ZZ$} is the closed linear span in $L^2_d(\Omega, \mathcal{F}, \PPP)$ of the past and present values of the process: $H^-_n := \spanot_d \set{\X_t : t \le n}$. The \emph{remote past} of $\{\X_t\}$ is $H_{-\infty} := \bigcap_{n \in \ZZ} H^-_n$. The process $\{\X_t\}$ is called \emph{regular}, if $H_{-\infty} = \{ \0 \}$ and it is called \emph{singular} if $H_{-\infty} = H(\X) := \spanot_d \set{\X_t : t \in \ZZ}$. Of course, there is a range of cases between these two extremes.

Singular processes are also called \emph{deterministic}, see e.g.\ Brockwell, Davis, and Fienberg (1991), because based on the past $H^-_0$, future values $\X_{1}$, $\X_{2}$, \dots, can be predicted with zero mean square error. On the other hand, regular processes are also called \emph{purely non-deterministic}, since their behavior are completely influenced by random \emph{innovations}. Consequently, knowing $H^-_0$, future values $\X_{1}$, $\X_{2}$, \dots, can be predicted only with positive mean square errors $\sigma^2_1$, $\sigma^2_2$, \dots, and $\lim_{t \to \infty} \sigma^2_t = \| \X_0 \|^2$. This shows why studying \emph{regular} time series is a primary target both in the theory and  applications. The Wold decomposition proves that any non-singular process can be decomposed into an orthogonal sum of a nonzero regular and a singular process. This also supports why it is important to study regular processes.

The Wold decomposition implies, see e.g.\ the classical references Rozanov (1967) and Wiener and Masani (1957), that $\{\X_t\}$ is regular if and only if $\{\X_t\}$ can be written as a causal infinite moving average
\begin{equation}\label{eq:causal_MA}
\X_t = \sum_{j=0}^{\infty} \bv(j) \xiv_{t-j}, \qquad t \in \ZZ, \qquad \bv(j) \in \CC^{d \times r} ,
\end{equation}
where $\{\xiv_t\}_{t \in \ZZ}$ is an $r$-dimensional $(r \le d)$ orthonormal white noise sequence: $\E \xiv_t = \0$, $\E(\xiv_s \xiv^*_t) = \delta_{s t} \I_r$, $\I_r$ is the $r \times r$ identity matrix. Equivalently,
\begin{equation}\label{eq:Wold_innov}
\X_t = \sum_{j=0}^{\infty} \av(j) \etav_{t-j}, \qquad t \in \ZZ, \qquad \av(j) \in \CC^{d \times d} ,
\end{equation}
where $\{\etav_t\}_{t \in \ZZ}$ is the $d$-dimensional white noise process of \emph{innovations}:
\begin{equation}\label{eq:innovation}
\etav_t := \X_t - \proj_{H^-_{t-1}} \X_t,  \quad t \in \mathbb{Z} ,
\end{equation}
$\E \etav_t = \0$, $\E(\etav_s \etav^*_t) = \delta_{s t} \Siga$, $\Siga$ is a $d \times d$ non-negative definite matrix of rank $r$.

It is also classical that any weakly stationary process has a non-negative definite spectral measure matrix $d\F$ on $[-\pi, \pi]$ such that
\[
\C(h) = \int_{-\pi}^{\pi} e^{i h \omega} d\F(\omega), \qquad h \in \ZZ.
\]
Then $\{\X_t\}$ is regular, see again e.g.\ Rozanov (1967) and Wiener and Masani (1957), if and only if $d\F  = \f$, the spectral density $\f$ has a.e.\ constant rank $r$, and can be factored in a form
\begin{equation}\label{eq:f_factor}
\f(\omega) = \frac{1}{2 \pi} \phiv(\omega) \phiv^*(\omega) , \quad \phiv(\omega) = [\phi_{k \ell}(\omega)]_{d \times r} , \quad \text{for a.e.\ } \omega \in [-\pi, \pi] ,
\end{equation}
where
\begin{equation}\label{eq:bj_Fourier}
\phiv(\omega) = \sum_{j=0}^{\infty} \tilde{\bv}(j) e^{-i j \omega}, \quad \|\phiv\|^2_2 = \sum_{j=0}^{\infty} \|\tilde{\bv}(j)\|^2_F < \infty,
\end{equation}
$\| \cdot \|_F$ denotes Frobenius norm. Consequently,
\begin{equation}\label{eq:Phiv}
\phiv(\omega) = \Phiv(e^{-i \omega}), \quad \omega \in (-\pi, \pi], \quad \Phiv(z) = \sum_{j=0}^{\infty} \tilde{\bv}(j) z^j,  \quad z \in D,
\end{equation}
so the entries of the \emph{spectral factor} $\Phiv(z) = [\Phi_{j k}(z)]_{d \times r}$ are  analytic functions in the open unit disc $D$ and belong to the class $L^2(T)$ on the unit circle $T$, consequently, they belong to the Hardy space $H^2$. Briefly, it is written as $\Phiv \in H^2$. Observe that there is a one-to-one correspondence between the interval $(-\pi, \pi]$ and the unit circle $T$ and, consequently, between functions defined on them.

Recall that the Hardy space $H^p$, $0 < p \le \infty$, denotes the space of all functions $f$ analytic in $D$ whose $L^p$ norms over all circles $\{z \in \CC : |z| = r\}$, $0 < r < 1$, are bounded, see e.g.\  Rudin (2006, Definition 17.7). If $p \ge 1$, then equivalently, $H^p$ is the Banach space of all functions $f \in L^p(T)$ such that
\[
f(e^{i \omega}) = \sum_{n=0}^{\infty} a_n e^{i n \omega} , \qquad \omega \in [-\pi, \pi],
\]
so the Fourier series of $f(e^{i \omega})$ is one-sided, $a_n = 0$ when $n < 0$, see e.g.\ Fuhrmann (2014, Section II.12). In particular, $H^2$ is a Hilbert space, which by Fourier transform isometrically isomorphic with the $\ell^2$ space of sequences $\{a_0, a_1, a_2, \dots\}$ with norm square
\[
\frac{1}{2 \pi} \int_{-\pi}^{\pi} |f(e^{i \omega})|^2 d\omega = \sum_{n=0}^{\infty} |a_n|^2 .
\]

For a one-dimensional time series $\{X_t\}$ ($d=1$) there exists a rather simple sufficient and necessary condition of regularity given by Kolmogorov (1941):
\begin{itemize}
  \item[(1)] $\{X_t\}$ has an absolutely continuous spectral measure with density $f$,
  \item[(2)] $\log f \in L^1$, that is, $\int_{-\pi}^{\pi} \log f(\omega) d\omega > -\infty$.
\end{itemize}
Then the Kolmogorov--Szeg\H{o} formula also holds:
\[
\sigma^2 = 2\pi \exp \int_{-\pi}^{\pi} \log f(\omega) \frac{d \omega}{2\pi}  ,
\]
where $\sigma^2$ is the variance of the innovations $\eta_t := X_t - \proj_{H^-_{t-1}} X_t$, that is, the variance of the one-step-ahead prediction.

For a multidimensional time series $\{\X_t\}$ which has \emph{full rank}, that is, when $\f$ has a.e.\ \emph{full rank}: $r=d$, and so the innovations $\etav_t$ defined by \eqref{eq:innovation} have full rank $d$, there exists a similar simple sufficient and necessary condition of regularity, see Rozanov (1967) and Wiener and Masani (1957):
\begin{itemize}
  \item[(1)] $\{\X_t\}$ has an absolutely continuous spectral measure matrix $d\F$ with density matrix $\f$,
  \item[(2)] $\log \det \f \in L^1$, that is $\int_{-\pi}^{\pi} \log \det \f(\omega) d\omega > -\infty$.
\end{itemize}
Then the $d$-dimensional Kolmogorov--Szeg\H{o} formula also holds:
\begin{equation}\label{eq:Kolm_Szego_mult}
\det \Siga = (2\pi)^d \exp \int_{-\pi}^{\pi} \log \det \bm{f}(\omega) \frac{d \omega}{2\pi}  ,
\end{equation}
where $\Siga$ is the covariance matrix of the innovations $\etav_t$ defined by \eqref{eq:innovation}.

On the other hand, the generic case of regular time series is more complicated. To my best knowledge, the result one has is Rozanov (1967, Theorem 8.1):
A $d$-dimensional stationary time series $\{\X_t\}$ is regular and of rank $r$, $1 \le r \le d$, if and only if each of the following conditions holds:
\begin{itemize}
  \item[(1)] It has an absolutely continuous spectral measure matrix $d\F$ with density matrix $\f(\omega)$ which has rank $r$ for a.e.\ $\omega \in[-\pi, \pi]$.
  \item[(2)] The density matrix $\f(\omega)$ has a principal minor $M(\omega) = \det [f_{i_p j_q}(\omega)]_{p,q=1}^r$, which is nonzero a.e.\ and
      \[
      \int_{-\pi}^{\pi} \log M(\omega) d\omega > - \infty.
      \]
   \item[(3)] Let $M_{k \ell}(\omega)$ denote the determinant obtained from $M(\omega)$ by replacing its $\ell$th row by the row $[f_{k i_p}(\omega)]_{p=1}^r$. Then the functions $\gamma_{k \ell}(e^{-i \omega}) = M_{k \ell}(\omega)/M(\omega)$ are boundary values of functions of the Nevanlinna class $N$.
\end{itemize}
It is immediately remarked in the cited reference that ``unfortunately, there is no general method determining, from the boundary value $\gamma(e^{-i \omega})$ of a function $\gamma(z)$, whether it belongs to the class $N$.''

Recall that the Nevanlinna class $N$ consists of all functions $f$ analytic in the open unit ball $D$ that can be written as a ratio $f = f_1/f_2$, $f_1 \in H^p$, $f_2 \in H^q$, $p, q > 0$, where $H^p$ and $H^q$ denote Hardy spaces, see e.g.\ Nikolski (1991, Definition 3.3.1).

The aim of this paper is to give a simpler, more usable sufficient condition to the regularity of generic weakly stationary time series. Also, this condition is used to show how regular processes satisfying these sufficient conditions can be approximated by a lower rank \emph{regular} process.

\section{Generic regular processes}\label{sse:generic_regular}

A simple idea is to use a spectral decomposition of the spectral density matrix to find an $H^2$ spectral factor if possible. (Take care that here we use the word `spectral' in two different meanings. On one hand, we use the spectral density of a time series in terms of a Fourier spectrum, on the other hand we take the spectral decomposition of a matrix in terms of eigenvalues and eigenvectors.)

So let $\{\X_t\}$ be a $d$-dimensional stationary time series and assume that its spectral measure matrix $d\F$ is absolutely continuous with density matrix $\f(\omega)$ which has rank $r$, $1 \le r \le d$, for a.e.\ $\omega \in [-\pi, \pi]$. Take the parsimonious spectral decomposition of the self-adjoint, non-negative definite matrix $\f(\omega)$:
\begin{equation}\label{eq:spectr_decomp_f_short1}
\f(\omega) = \sum_{j=1}^{r} \lambda_j(\omega) \uu_j(\omega) \uu^*_j(\omega) = \tilde{\UU}(\omega) \Lada_r(\omega) \tilde{\UU}^*(\omega) ,
\end{equation}
where
\begin{equation}\label{eq:Lam_r}
\Lada_r(\omega) = \diag[\lambda_1(\omega), \dots , \lambda_r(\omega)], \quad \lambda_1(\omega) \ge \cdots \ge \lambda_r(\omega) > 0, \quad \text{a.e.} \quad \omega \in [-\pi, \pi ],
\end{equation}
is the diagonal matrix of nonzero eigenvalues of $\f(\omega)$ and
\[
\tilde{\UU}(\omega) = [\uu_1(\omega), \dots , \uu_r(\omega)] \in \CC^{d \times r}
\]
is a sub-unitary matrix of corresponding right eigenvectors, not unique even if all eigenvalues are distinct. Then, still, we have
\begin{equation}\label{eq:spectr_decomp_f_short}
\Lam_r(\omega) = \tilde{\UU}^*(\omega) \f(\omega) \tilde{\UU}(\omega) .
\end{equation}

The matrix function $\Lam_r(\omega)$ is a self-adjoint, positive definite function, $\tr(\Lam_r(\omega))$ $= \tr(\f(\omega))$, where $\f(\omega)$ is the density function of a finite spectral measure. This shows that the integral of $\tr(\Lam_r(\omega))$ over $[-\pi, \pi]$ is finite. Thus $\Lam_r(\omega)$ can be considered the spectral density function of a full rank regular process. (See the details below, in the proof of Theorem \ref{th:generic_reg_nec}.) So it can be factored, in fact, we may take a maximal $H^2$ spectral factor $\DD_r(\omega)$ of it:
\begin{equation}\label{eq:Lamr_factor1}
\Lam_r(\omega) = \frac{1}{2 \pi} \DD_r(\omega) \DD_r(\omega),
\end{equation}
where $\DD_r(\omega)$ is a diagonal matrix.

Then a simple way to factorize $\f$ is to choose
\begin{equation}\label{eq:trivial_spectral_factor}
\phiv(\omega) = \tilde{\UU}(\omega) \DD_r(\omega) \A(\omega) = \tilde{\UU}(\omega) \A(\omega) \DD_r(\omega) = \tilde{\UU}_{\A}(\omega) \DD_r(\omega),
\end{equation}
where $\A(\omega) = \diag[a_1(\omega), \dots , a_r(\omega)]$, each $a_k(\omega)$ being a measurable function on $[-\pi, \pi]$ such that $|a_k(\omega)| = 1$ for any $\omega$, but otherwise arbitrary and $\tilde{\UU}_{\A}(\omega)$ still denotes a sub-unitary matrix of eigenvectors of $\f$ in the same order as the one of the eigenvalues.

Presently I do not know if for any regular time series $\{\X_t\}$ there exists a matrix valued function $\A(\omega)$ such that $\phiv(\omega)$ defined by \eqref{eq:trivial_spectral_factor} has a Fourier series with only non-negative powers of $e^{-i \omega}$ or not. Equivalently, does there exist an $H^2$ spectral factor $\Phiv(z)$  whose boundary value is the above $\phiv(\omega)$ with some $\A(\omega)$, according to the formulas \eqref{eq:f_factor}--\eqref{eq:Phiv}?

A useful consequence of the fact if the regular time series $\{\X_t\}$ is such that does have an $H^2$ spectral factor of the form \eqref{eq:trivial_spectral_factor} is that then one can easily find regular low rank approximations with rather easily computable mean square errors, see Subsection \ref{sse:low_regular}.

\begin{thm}\label{th:generic_reg_nec}

\begin{itemize}
\item[(a)] Assume that a $d$-dimensional stationary time series $\{\X_t\}$ is regular of rank $r$, $1 \le r \le d$. Then for $\Lam_r(\omega)$ defined by \eqref{eq:Lam_r} one has $\log \det \Lam_r \in L^1 = L^1([-\pi, \pi], \mathcal{B}, d\omega)$, equivalently,
\begin{equation}\label{eq:log_lambda}
    \int_{-\pi}^{\pi} \log \lambda_r(\omega) \, d\omega > -\infty .
\end{equation}

\item[(b)] If moreover one assumes that the regular time series $\{\X_t\}$ is such that has an $H^2$ spectral factor of the form \eqref{eq:trivial_spectral_factor}, then the following statement holds as well:

The sub-unitary matrix function $\tilde{\UU}(\omega)$ appearing in the spectral decomposition of $\f(\omega)$ in \eqref{eq:spectr_decomp_f_short1} can be chosen so that it belongs to the Hardy space $H^{\infty} \subset H^2$, thus

\begin{equation}\label{eq:UU_Fourier_series}
  \tilde{\UU}(\omega) = \sum_{j=0}^{\infty} \psiv(j) e^{-i j \omega} , \quad \psiv(j) \in \CC^{d \times r}, \quad \sum_{j=0}^{\infty} \|\psiv(j)\|^2_F < \infty .
\end{equation}
In this case one may call $\tilde{\UU}(\omega)$ a matrix inner function.
\end{itemize}
\end{thm}
\begin{proof}
Assume first that $\{\X_t\}$ is regular and of rank $r$, $1 \le r \le d$. Here we are going to use the Wold decomposition. Then it follows that $\{\X_t\}$ has an absolutely continuous spectral measure with density matrix $\f(\omega)$, which has constant rank $r$ for a.e.\ $\omega$.

In order to check that $\log \det \Lam_r \in L^1$, it is enough to show that \eqref{eq:log_lambda} holds. First,
\[
\int_{-\pi}^{\pi} \log \det \Lam_r(\omega) d \omega < \infty
\]
is always true. Namely, using the inequality $\log x < x$ if $x > 0$, we see that
\[
\log \det \Lam_r(\omega) = \sum_{j=1}^r \log \lambda_j(\omega) < \tr(\Lam_r(\omega)) = \tr(\f(\omega)) \in L^1,
\]
since the spectral density $\f$ is an integrable function. Second, by \eqref{eq:Lam_r}, \begin{equation}\label{eq:log_det_lam_all}
\int_{-\pi}^{\pi} \sum_{j=1}^r \log \lambda_j(\omega) d \omega > -\infty \quad \Leftrightarrow \quad \int_{-\pi}^{\pi} \log \lambda_r(\omega) \, d\omega > -\infty .
\end{equation}

By formulas \ref{eq:f_factor}--\ref{eq:Phiv} it follows that
\begin{equation}\label{eq:factorization}
\f(\omega) = \frac{1}{2 \pi} \Phiv(e^{-i \omega}) \Phiv^*(e^{-i \omega}),
\end{equation}
where the spectral factor $\Phiv(z) = [\Phi_{jk}(z)]_{d \times r}$ is an analytic function in the open unit disc $D$ and $\Phiv \in H^2$. Equality \eqref{eq:factorization} implies that every principal minor $M(\omega) = \det [f_{j_p j_q}]_{p, q = 1}^r$  of $\f$ is a constant times the product of a minor $M_{\Phiv}(e^{-i \omega})$ of $\Phiv(e^{-i \omega})$ and its conjugate:
\begin{align}\label{eq:principal_minor}
M(\omega) &= (2 \pi)^{-r} \det [\Phi_{j_p k}(e^{-i \omega})]_{p,k=1}^r \det [\overline{\Phi_{j_p k}(e^{-i \omega})}]_{p,k=1}^r \nonumber \\
&= (2 \pi)^{-r} \left|\det [\Phi_{j_p k}(e^{-i \omega})]_{p,k=1}^r\right|^2 =: (2 \pi)^{-r} |M_{\Phiv}(e^{-i \omega})|^2 .
\end{align}
The row indices of the minor $M_{\Phiv}(z)$ in the matrix $\Phiv(z)$ are the same indices $j_p$, $p = 1, \dots , r$, that define the principal minor $M(\omega)$ in the matrix $\f(\omega)$. Since the function $M_{\Phiv}(z) = \det [\Phi_{j_p k}(z)]_{p, k = 1}^r$ is analytic in $D$, it is either identically zero or is different from zero a.e. The rank of $\f$ is $r$ a.e., therefore the sum of all its principal minors of order $r$ (which are non-negative since $\f$ is non-negative definite) must be different from zero a.e. The last two sentences imply that there exists a principal minor $M(\omega)$ of order $r$ which is different from zero a.e. We are using this principal minor $M(\omega)$ from now on.

The entries of $\Phiv(z)$ are in $H^2$, so e.g.\ by Wiener and Masani(1957, Lemma 3.7) it follows that the determinant $M_{\Phiv}(z) \in H^{2/r}$. Then e.g.\ Rudin (2006, Theorem 17.17) implies that $\log |M_{\Phiv}(e^{-i \omega})| \in L^1$, which in turn with \eqref{eq:principal_minor} imply that
\begin{align}\label{eq:logdetLam_r}
\log M(\omega) = \log \left\{(2 \pi)^{-r} \left|M_{\Phiv}(e^{-i \omega})\right|^2\right\}
= -r \log 2 \pi + 2 \log \left|M_{\Phiv}(e^{-i \omega})\right| \in L^1 .
\end{align}

Further, let us define the corresponding minor of the matrix $\tilde{\UU}(\omega)$ by $M_{\tilde{\UU}}(\omega) := \det[\tilde{U}_{j_p k}]_{p,k=1}^r$. Since $\tilde{\UU}(\omega)$ is a sub-unitary matrix, its each entry has absolute value less than or equal to $1$. Consequently, $|M_{\tilde{\UU}}(\omega)| \le r!$. By \eqref{eq:spectr_decomp_f_short1},
\[
M(\omega) = M_{\tilde{\UU}}(\omega) \det \Lam_r(\omega) \overline{M_{\tilde{\UU}}(\omega)} = \det \Lam_r(\omega) \, |M_{\tilde{\UU}}(\omega)|^2 .
\]
It follows that
\[
\log \det \Lam_r(\omega) = \log M(\omega) - 2 \log |M_{\tilde{\UU}}(\omega)| \ge \log M(\omega) - 2 \log(r!),
\]
which with \eqref{eq:logdetLam_r} and \eqref{eq:log_det_lam_all} shows \eqref{eq:log_lambda}, and this proves part (a) of the theorem.

As was mentioned above before this theorem, the matrix function $\Lam_r(\omega)$ is a self-adjoint, positive definite function, it can be considered as the spectral density function of an $r$-dimensional stationary time series $\{\V_t\}_{t \in \Z}$ of full rank $r$. In fact, by linear filtering we define
\begin{equation}\label{eq:Vt_filter_def}
\V_t := \int_{-\pi}^{\pi} e^{i t \omega} \tilde{\UU}^*(\omega) d\Z_{\omega}, \quad t \in \mathbb{Z}.
\end{equation}
Then its auto-covariance function and spectral density are really given by
\begin{equation}\label{eq:Vtprocess}
\C_{\V}(h) = \int_{-\pi}^{\pi} e^{i h \omega} \tilde{\UU}^*(\omega) \f(\omega) \tilde{\UU}(\omega) d\omega = \int_{-\pi}^{\pi} e^{i h \omega} \Lam_r(\omega) d\omega, \quad h \in \ZZ .
\end{equation}
The classical case of full rank regular processes described in Section 1 and condition (a) of the present theorem show that $\{\V_t\}$ is a regular time series.

Thus one can take a maximal $H^2$ spectral factor $\DD_r(\omega)$ by formula  \eqref{eq:Lamr_factor1}. It follows that
\[
 \DD_r(\omega) = \sum_{j=0}^{\infty} \deltav(j) e^{-i j \omega}, \quad \sum_{j=0}^{\infty} \|\deltav(j)\|_F^2 < \infty,
\]
and
\begin{equation}\label{eq:Lamr_factor3}
\Deltav(z):= \sum_{j=0}^{\infty} \deltav(j) z^j, \quad z \in D, \quad
\Deltav(e^{-i \omega}) = \DD_r(\omega),
\end{equation}
$\Deltav(z)$ is a maximal $H^2$ spectral factor. Since the non-diagonal entries of the boundary value $\DD_r(\omega)$ are zero, it follows that the non-diagonal entries of $\Deltav(z)$ are also zero, see the integral formulas in Rudin (2006, Theorem 17.11). Since the diagonal entries of $\Lam_r(\omega)$ are positive, it follows that the diagonal entries of $\DD_r(\omega) = \Deltav(e^{-i \omega})$ are nonzero. The components of the process $\{\V_t\}$ are regular 1D time series, so by Kolmogorov (1941, Theorme 21) the $H^2$-functions $\Delta_{kk}(z)$ $(k=1, \dots , r)$ have no zeros in the open unit disc $D$.

Now assume that the spectral density $\f$ has an $H^2$ spectral factor $\phiv(z)$ of the form \eqref{eq:trivial_spectral_factor}, $\phiv(\omega) = \Phiv(e^{-i \omega})$. It means that $\Phiv(e^{-i \omega}) = \tilde{\UU}(\omega) \Deltav(e^{-i \omega})$,
\[
\tilde{\UU}(\omega) = \Phiv(e^{-i \omega}) \left[
                                      \begin{array}{ccc}
                                        \Delta_{11}^{-1}(e^{-i \omega}) & \cdots & 0 \\
                                        \vdots & \ddots & \vdots \\
                                        0 & \cdots & \Delta_{rr}^{-1}(e^{-i \omega}) \\
                                      \end{array}
                                    \right] .
\]
 Consequently, the entries on the right hand side are boundary values of the ratio of two  $H^2$-functions, and the denominator has no zeros in $D$. Hence $\tilde{\UU}(\omega)$ is the boundary value of an analytic function $\WW(z) = [w_{k \ell}(z)]_{d \times r}$ defined in $D$:
\begin{equation}\label{eq:U_W1}
\tilde{\UU}(\omega) = \WW(e^{-i \omega}) .
\end{equation}
Moreover, since the boundary value $\UU(\omega)$ is unitary, its entries are bounded functions. It implies that
\begin{align}\label{eq:U_W2}
\WW(z) &= \sum_{j=0}^{\infty} \psiv(j) z^j \in H^{\infty} \subset H^2 , \nonumber \\
\tilde{\UU}(\omega) &= \sum_{j=0}^{\infty} \psiv(j) e^{-i j \omega} , \quad \sum_{j=0}^{\infty} \|\psiv(j)\|^2_F < \infty .
\end{align}
This proves part (b) of the theorem.
\end{proof}

The next theorem gives a sufficient condition for the regularity of a generic weakly stationary time series; compare with the statements of Theorem \ref{th:generic_reg_nec} above. Observe that assumptions (1) and (2) in the next theorem are necessary conditions of regularity as well. Only assumption (3) is not known to be necessary. We think that these assumptions are simpler to check in practice then the ones of Rozanov's theorem cited above. By formula \eqref{eq:trivial_spectral_factor}, checking assumption (3) means that for each  eigenvectors $\uu_k(\omega)$ of $\f(\omega)$ we are searching for a complex function multiplier $a_k(\omega)$ of unit absolute value that gives an $H^{\infty}$ function result.

\begin{thm}\label{th:generic_reg_suf}
Let $\{\X_t\}$ be a $d$-dimensional time series. It is regular of rank $r \le d$ if the following three conditions hold.
\begin{itemize}

\item[(1)]
It has an absolutely continuous spectral measure matrix $d\F$ with density matrix $\f(\omega)$ which has rank $r$ for a.e.\ $\omega \in[-\pi, \pi]$.

\item[(2)]
For $\Lam_r(\omega)$ defined by \eqref{eq:Lam_r} one has $\log \det \Lam_r \in L^1 = L^1([-\pi, \pi], \mathcal{B}, d\omega)$, equivalently, \eqref{eq:log_lambda} holds.

\item[(3)]
The sub-unitary matrix function $\tilde{\UU}(\omega)$ appearing in the spectral decomposition of $\f(\omega)$ in \eqref{eq:spectr_decomp_f_short1} can be chosen so that it belongs to the Hardy space $H^{\infty} \subset H^2$, thus \eqref{eq:UU_Fourier_series} holds.

\end{itemize}
\end{thm}
\begin{proof}
Assume that conditions (1), (2), and (3) of the theorem hold. Conditions (1) and (2) give that $\Lam_r(\omega)$ is the spectral density of an $r$-dimensional regular stationary time series $\{\V_t\}_{t \in \mathbb{Z}}$ of full rank $r$, just like in the proof of Theorem \ref{th:generic_reg_nec}. Then $\Lam_r(\omega)$ has the factorization \eqref{eq:Lamr_factor1}, \eqref{eq:Lamr_factor3}. Condition (3) implies that $\tilde{\UU}(\omega) \in H^{\infty} \subset H^2$, with properties \eqref{eq:U_W1} and \eqref{eq:U_W2}. The spectral decomposition \eqref{eq:spectr_decomp_f_short1} can be written as
\[
\f(\omega) = \tilde{\UU}(\omega) \Lam_r(\omega) \tilde{\UU}^*(\omega) = \frac{1}{2 \pi} \tilde{\UU}(\omega) \DD_r(\omega) \DD_r(\omega) \tilde{\UU}^*(\omega) = \frac{1}{2 \pi} \phiv(\omega) \phiv^*(\omega).
\]
So
\begin{align*}
\phiv(\omega) &= \tilde{\UU}(\omega) \DD_r(\omega) =  \sum_{j=0}^{\infty} \psiv(j) e^{-i j \omega} \sum_{k=0}^{\infty} \deltav(j) e^{-i k \omega} = \sum_{\ell=0}^{\infty} \bv(\ell) e^{-i \ell \omega}, \\
\bv(\ell) &= \sum_{j=0}^{\ell} \psiv(j) \deltav(\ell - j)  .
\end{align*}
These imply that
\[
\Phiv(z) := \sum_{\ell=0}^{\infty} \bv(\ell) z^{\ell} = \WW(z) \Deltav(z), \quad z \in D ,
\]
where $\WW \in H^{\infty}$ and $\Deltav \in H^2$, thus $\Phiv(z) \in H^2$. By formulas \ref{eq:f_factor}--\ref{eq:Phiv} it means that the time series $\{\X_t\}$ is regular. This completes the proof of the theorem.
\end{proof}

\begin{rem}\label{re:cond_3 full_rank}
Comparing the full rank case and Theorem \ref{th:generic_reg_suf} shows that in the full rank case, conditions (1) and (2) are sufficient without any further assumption.
\end{rem}

\begin{rem}\label{re:Vtprocess}
Assume that $\{\X_t\}$ is a $d$-dimensional regular time series of rank $r$.
Assume as well that its spectral density matrix $\f$ has the spectral decomposition \eqref{eq:spectr_decomp_f_short1}. Then the $r$-dimensional time series $\{\V_t\}$ can be given by \eqref{eq:Vt_filter_def}, whose spectral density by \eqref{eq:Vtprocess} is a diagonal matrix $\f_{\V} = \Lam_r$ and, consequently, its covariance matrix function is also diagonal $\C_{\V}(h)$, $h \in \ZZ$. It means that the process $\{\V_t\}$ is cross-sectionally orthogonal. Its orthogonal components are the \emph{Dynamic Principal Components (DPC)} of the original process $\{\X_t\}$.

Assuming condition (3) of Theorem \ref{th:generic_reg_suf}, by the definition of $\{\V_t\}$ in \eqref{eq:Vt_filter_def} and by \eqref{eq:U_W2} it follows that
\[
\tilde{\UU}^*(\omega) = \sum_{j=0}^{\infty} \psiv^*(j) e^{i j \omega}, \qquad \sum_{j=0}^{\infty} \|\psiv^*(j)\|^2_F < \infty ,
\]
and
\begin{equation}\label{eq:Vt_MA}
\V_t = \int_{-\pi}^{\pi} e^{i t \omega} \sum_{j=0}^{\infty} \psiv^*(j) e^{i j \omega} d\Z_{\omega} = \sum_{j=0}^{\infty} \psiv^*(j) \int_{-\pi}^{\pi} e^{i (t+j) \omega} d\Z_{\omega}
= \sum_{j=0}^{\infty} \psiv^*(j) \X_{t+j}.
\end{equation}
It is interesting that $\{\V_t\}$ is \emph{not causally subordinated} to $\{\X_t\}$. In fact, $\{\V_t\}$ is \emph{`anti-causal'} in terms of $\{\X_t\}$: depends only on its present and future terms. Its necessity follows from formulas \eqref{eq:Vt_filter_def} and \eqref{eq:Vtprocess}: we need $\tilde{\UU}^*(\omega)$ in the definition of the process $\V_t$ if we want cross-sectionally orthogonal components and the Fourier series of $\tilde{\UU}^*(\omega)$ contains only non-positive powers of $e^{-i \omega}$.

Equation \eqref{eq:Vt_MA} shows that the $r$-dimensional DPC process $\{\V_t\}$ can be obtained from the original process $\{\X_t\}$ as a linear transform. Now we are going to show that the DPC process contains all essential information about the original process, since $\{\X_t\}$ can be obtained from $\{\V_t\}$ by a (causal) linear transform. If the rank $r$ is significantly smaller than the dimension $d$, then this fact can be used for \emph{dimension reduction}.

Using \eqref{eq:U_W2}, define
\begin{align*}
\Y_t &:= \int_{-\pi}^{\pi} e^{i t \omega} \tilde{\UU}(\omega) d\Z^{\V}_{\omega} = \int_{-\pi}^{\pi} e^{i t \omega} \sum_{j=0}^{\infty} \psiv(j) e^{-i j \omega} d\Z^{\V}_{\omega} \\
&= \sum_{j=0}^{\infty} \psiv(j) \int_{-\pi}^{\pi} e^{i (t-j) \omega} d\Z^{\V}_{\omega} = \sum_{j=0}^{\infty} \psiv(j) \V_{t-j}, \qquad t \in \ZZ .
\end{align*}
Then substituting $\f = \tilde{\UU} \Lam_r \tilde{\UU}^*$, we obtain
\begin{align*}
\| \X_t - \Y_t \|^2 &= \tr \; \cov \left\{(\X_t - \Y_t), (\X_t - \Y_t) \right\} \\
&= \tr \int_{-\pi}^{\pi} (\I_d - \tilde{\UU} \tilde{\UU}^*) \, \f \, (\I_d - \tilde{\UU} \tilde{\UU}^*) \, d \omega  \\
&= \tr \int_{-\pi}^{\pi} (\tilde{\UU}_r \Lam_r \tilde{\UU}_r^* -  \tilde{\UU}_r \Lam_r \tilde{\UU}_r^* -  \tilde{\UU}_r \Lam_r \tilde{\UU}_r^* +  \tilde{\UU}_r \Lam_r \tilde{\UU}_r^*)\, d \omega = 0 ,
\end{align*}
for any $t \in \ZZ$. This shows that $\Y_t = \X_t$ a.s.\, for any $t \in \ZZ$. Thus we have shown that
\[
\X_t = \sum_{j=0}^{\infty} \psiv(j) \V_{t-j}, \quad t \in \Z, \quad \psiv(j) = \frac{1}{2 \pi} \int_{-\pi}^{\pi} \tilde{\UU}(\omega) e^{i j \omega} d\omega ,
\]
where $\{\V_t\}$ is an $r$-dimensional cross-sectionally orthogonal process:
\[
\C_{\V}(h) = \cov(\V_{j+h}, \V_j) = \int_{-\pi}^{\pi} e^{i h \omega} \Lam_r(\omega) d\omega, \qquad h \in \ZZ.
\]
In particular,
\[
\cov(V^k_j, V^{\ell}_j) = \delta_{k \ell} \int_{-\pi}^{\pi} \lambda_k(\omega) d\omega, \qquad k=1, \dots , r.
\]

\end{rem}

\begin{rem}
This remark shows that condition (3) of Theorem \ref{th:generic_reg_suf} is natural for a regular process. To see this we compare the sufficient and necessary conditions of regularity described in \eqref{eq:causal_MA} and \eqref{eq:f_factor}, respectively, with condition (3) of Theorem \ref{th:generic_reg_suf}.

By \eqref{eq:f_factor} and \eqref{eq:spectr_decomp_f_short1},
\[
\f(\omega) = \frac{1}{2 \pi} \phiv(\omega) \phiv^*(\omega) = \tilde{\UU}(\omega) \Lada_r(\omega) \tilde{\UU}^*(\omega) .
\]
Because of the one-sided Fourier series \eqref{eq:bj_Fourier} of $\phiv( \omega)$ and \eqref{eq:U_W2} of $\tilde{\UU}(\omega)$ and also by \eqref{eq:Lamr_factor1}, we may write that
\[
\phiv(\omega) = \sqrt{2 \pi} \, \tilde{\UU}(\omega) \Lam_r^{\frac12}(\omega), \qquad \Lam_r^{\frac12}(\omega) = \frac{1}{\sqrt{2 \pi}} \DD_r(\omega) = \frac{1}{\sqrt{2 \pi}} \sum_{j=0}^{\infty} \deltav(j) e^{i j \omega},
\]
so we have
\[
\sum_{j=0}^{\infty} \bv(j) e^{-i j \omega} = \sum_{k, \ell = 0}^{\infty} \psiv(k) \deltav(\ell) e^{-i (k + \ell) \omega} , \quad \omega \in [-\pi, \pi] .
\]
It follows that
\[
\bv(j) = \sum_{k=0}^j \psiv(k) \deltav(j-k), \qquad j \ge 0 .
\]
Comparing with \eqref{eq:causal_MA}, it implies
\[
\X_t = \sum_{j=0}^{\infty} \bv(j) \xiv_{t-j} = \sum_{j=0}^{\infty} \sum_{k=0}^j \psiv(k) \deltav(j-k) \xiv_{t-j}, \quad t \in \ZZ,
\]
where $\{\xiv_t\}$ is a standard white noise sequence.
\end{rem}

\begin{cor}\label{co:Szego_Kolm_gen_mult}
Assume that $\{\X_t\}$ is a $d$-dimensional regular stationary time series of rank $r$, $1 \le r \le d$. Then a Kolmogorov--Szeg\H{o} formula holds:
\[
\det \Siga_r = (2\pi)^r \exp \int_{-\pi}^{\pi} \log \det \Lam_r(\omega) \frac{d \omega}{2\pi} = (2\pi)^r \exp \int_{-\pi}^{\pi} \sum_{j=1}^r \log \lambda_j(\omega) \frac{d \omega}{2\pi} ,
\]
where $\Lam_r$ is defined by \eqref{eq:Lam_r} and $\Siga_r$ is the covariance matrix of the innovation process of an $r$-dimensional subprocess $\{\X_t^{(r)}\}$ of rank $r$, as defined below in the proof.
\end{cor}
\begin{proof}
Let $\f_r(\omega)$ be the submatrix of $\f(\omega)$ whose determinant $\M(\omega)$ was defined in the first part of the proof of Theorem \ref{th:generic_reg_nec}:
\[
\f_r(\omega) = [f_{j_p j_q}(\omega)]_{p,q = 1}^r , \quad \det \f_r(\omega) = M(\omega) \ne 0 \quad \text{ for a.e.} \quad \omega \in [-\pi, \pi].
\]
The indices $j_p$, $p = 1, \dots , r$, define a subprocess $\X_t^{(r)} = [X_t^{j_1}, \dots , X_t^{j_r}]^T$ of the original time series $\{\X_t\}$. Then $\{\X_t^{(r)}\}$ has an absolutely continuous spectral measure with density $\f_r(\omega)$, and by \eqref{eq:logdetLam_r},
\[
\log \det \f_r = \log M  \in L^1 .
\]
Hence by full rank case \eqref{eq:Kolm_Szego_mult}, $\{\X_t^{(r)}\}$ is a regular process of full rank $r$ and
\[
\det \Siga_r = (2\pi)^r \exp \int_{-\pi}^{\pi} \log \det \f_r(\omega) \frac{d \omega}{2\pi} ,
\]
where $\Siga_r$ is the covariance matrix of the innovation process of $\{\X_t^{(r)}\}$ as defined by \eqref{eq:innovation}:
\[
\etav_t^{(r)} := \X_t^{(r)} - \mathrm{Proj}_{H^-_{t-1}} \X_t^{(r)} \quad (t \in \ZZ), \quad \Siga_r := \E \left(\etav_0^{(r)} \etav_0^{(r)*}\right) .
\]
Here we used that the past until $(t-1)$ of the subprocess $\{\X_t^{(r)}\}$ is the same as the past $H^-_{t-1}$ of $\{\X_t\}$. Really, by \eqref{eq:Wold_innov}, for the regular process $\{\X_t\}$ of rank $r$ we have a causal MA($\infty$) form:
\[
\X_t = \sum_{j=0}^{\infty} \av(j) \etav_{t-j}, \quad t \in \ZZ,
\]
and for each $t$, $\etav_t$ is linearly dependent on $\etav_t^{(r)}$. Thus
\[
\spanot \{\X^{(r)}_{t-1}, \X^{(r)}_{t-2}, \X^{(r)}_{t-3}, \dots \} = \spanot \{\X_{t-1}, \X_{t-2}, \X_{t-3}, \dots \} = H^-_{t-1} .
\]
\end{proof}


\subsection{Classification of non-regular multidimensional time series}\label{sse:class_non-regular}

We call a time series \emph{non-regular} if either it is singular or its Wold decomposition contains two orthogonal, non-vanishing processes: a regular and a singular one. The classification below follows from Theorem \ref{th:generic_reg_suf}, namely, which of its conditions (1), (2), and (3) is `the first one to be violated.'

In dimension $d > 1$ a non-regular process beyond its regular part may have a singular part with non-vanishing spectral density. For example, if $d > 1$ and $\X_t = \Y_t + \Z_t$, $\Y_t \perp \Z_t$, $t \in \ZZ$, it is possible that $\{\Y_t\}$ is regular and $\{\Z_t\}$ is singular, both with a non-vanishing spectral density.

Below we are considering a $d$-dimensional stationary time series $\{\X_t\}$ with spectral measure $d\F$.

\begin{itemize}
  \item \emph{Type (0) non-regular processes.} In this case the spectral measure $d\F$ of the time series $\{\X_t\}$ is singular w.r.t.\ the Lebesgue measure in $[-\pi, \pi]$. Clearly, type (0) non-regular processes are simply singular ones. We may further divide this class into processes with a discrete spectrum or processes with a continuous singular spectrum or processes with both.
  \item \emph{Type (1) non-regular processes.} The time series has an absolutely continuous spectral measure with density $\f$, but $\mathrm{rank}(\f)$ is not constant. It means that there exist measurable subsets $A, B \subset[-\pi, \pi]$ such that $d\omega(A) > 0$ and $d\omega(B) > 0$, $\mathrm{rank}(\f(\omega)) = r_1$ if $\omega \in A$, $\mathrm{rank}(\f(\omega)) = r_2$ if $\omega \in B$, and $r_1 \ne r_2$. Here $d\omega$ denotes the Lebesgue measure in $[-\pi, \pi]$.
  \item \emph{Type (2) non-regular processes.} The time series has an absolutely continuous spectral measure with density $\f$ which has constant rank $r$ a.e., $1 \le r \le d$, but
\[
\int_{-\pi}^{\pi} \log \det \Lam_r(\omega) d \omega = \int_{-\pi}^{\pi} \sum_{j=1}^r \log \lambda_j(\omega) \, d\omega = -\infty ,
\]
where $\Lam_r$ is defined by \eqref{eq:Lam_r}.
  \item \emph{Type (3) non-regular processes.} The time series has an absolutely continuous spectral measure with density $\f$ which has constant rank $r$ a.e., $1 \le r < d$,
\[
\int_{-\pi}^{\pi} \log \det \Lam_r(\omega) d \omega = \int_{-\pi}^{\pi} \sum_{j=1}^r \log \lambda_j(\omega) \, d\omega > -\infty ,
\]
but the unitary matrix function $\tilde{\UU}(\omega)$ appearing in the spectral decomposition of $\f(\omega)$ in \eqref{eq:spectr_decomp_f_short} cannot be defined so that it belongs to the Hardy space $H^2$.

Because it seems unknown if condition (3) of Theorem \ref{th:generic_reg_suf} is a necessary condition of regularity or not, this situation in itself does not necessarily show that the process is non-regular.
\end{itemize}

It is worth mentioning that if $\{\X_t\}$ has full rank $r=d$ and it is non-singular, then its singular part may have only a spectral measure which is singular w.r.t.\ Lebesgue measure, see Wiener and Masani (1957, Theorem 7.11).


\subsection{Examples}\label{sse:examples}

In this subsection simple examples are given for non-regular processes described in the previous subsection, and also for a regular process satisfying the conditions of Theorem \ref{th:generic_reg_suf}. Each example will have dimension $d=3$ and rank $r(\omega) \le 2$. For simplicity, in the first three examples each process will have the form $\X_t = [X^1_t, X^2_t, X^3_t]^T$, $X^3_t = X^1_t + X^2_t$, $X^1_s \perp X^2_t$, $s, t \in \ZZ$. The covariance matrix function of the process, equivalently, will be
\begin{equation}\label{eq:exmpl_C}
\C(h) = \left[
          \begin{array}{ccc}
            c_{11}(h) & 0 & c_{11}(h) \\
            0& c_{22}(h) & c_{22}(h)\\
            c_{11}(h) & c_{22}(h) & c_{11}(h)+c_{22}(h) \\
          \end{array}
        \right], \quad h \in \ZZ .
\end{equation}
Also, equivalently, the spectral measure matrix will have the form
\begin{equation}\label{eq:exmpl_F}
d\F(\omega) = \left[
                \begin{array}{ccc}
                  dF^{11}(\omega) & 0 & dF^{11}(\omega) \\
                  0 & dF^{22}(\omega) & dF^{22}(\omega) \\
                  dF^{11}(\omega) & dF^{22}(\omega) & dF^{11}(\omega) + dF^{22}(\omega) \\
                \end{array}
              \right], \quad \omega \in [-\pi, \pi].
\end{equation}
In particular, if $d\F$ is absolutely continuous with density $\f$, then \eqref{eq:exmpl_F} is replaced by
\begin{equation}\label{eq:exmpl_f}
\f(\omega) = \left[
               \begin{array}{ccc}
                 f^{11}(\omega) & 0 & f^{11}(\omega) \\
                 0 & f^{22}(\omega) & f^{22}(\omega) \\
                 f^{11}(\omega) & f^{22}(\omega) & f^{11}(\omega) + f^{22}(\omega) \\
               \end{array}
             \right].
\end{equation}
We will need the eigenvalues of $\f(\omega)$:
\begin{align*}
\det(\f - \lambda \I_3) &=
\left|
\begin{array}{ccc}
f^{11}(\omega) - \lambda & 0 & f^{11}(\omega) \\
0 & f^{22}(\omega) - \lambda & f^{22}(\omega) \\
f^{11}(\omega) & f^{22}(\omega) & f^{11}(\omega) + f^{22}(\omega) - \lambda \\
\end{array}
\right| \\
&=
\left|
\begin{array}{ccc}
f^{11}(\omega) - \lambda & 0 & f^{11}(\omega) \\
0 & f^{22}(\omega) - \lambda & f^{22}(\omega) \\
\lambda & \lambda & -\lambda \\
\end{array}
\right| \\
&= -\lambda [\lambda^2 - 2 \lambda (f^{11}(\omega) + f^{22}(\omega)) + 3 f^{11}(\omega) f^{22}(\omega)] .
\end{align*}

Thus
\[
\lambda_{1,2}(\omega) = f^{11}(\omega) + f^{22}(\omega) \pm \sqrt{f^{11}(\omega)^2 + f^{22}(\omega)^2 - \f^{11}(\omega) \f^{22}(\omega)},
\]
and $\lambda_3(\omega) = 0$. It shows that $\rk(\f(\omega)) \le 2$, and
\begin{equation}\label{eq:det_Lam}
\det \Lam_2(\omega) = \lambda_1(\omega) \lambda_2(\omega) = 3 f^{11}(\omega) f^{22}(\omega) .
\end{equation}

\subsubsection{A type (0) non-regular process}
Let $n \ge 1$ and
\[
\X_t = \sum_{j=1}^n \A_j e^{i t \omega_j}, \qquad t \in\ZZ,
\]
where $\omega_j \in (-\pi, \pi]$ are distinct frequencies and $\A_j$ are pairwise orthogonal random amplitudes, $j=1, \dots, n$. We assume that $\A_j = [A^1_j, A^2_j, A^3_j]^T$, $A^3_j = A^1_j + A^2_j$, $A^1_j \perp A^2_j$, $A^{\ell}_j \sim \mathcal{N}(0, v^{\ell}_j)$, $v^{\ell}_j > 0$, where $\ell=1,2$ and $j=1, \dots, n$.

Then
\[
\C(h) = \EE(\X_{t+h} \X_t^*) = \sum_{j=1}^n \EE(\A_j \A_j^*) e^{i h \omega_j} = \sum_{j=1}^n  e^{i h \omega_j} \left[
                   \begin{array}{ccc}
                     v^1_j & 0 & v^1_j \\
                     0 & v^2_j & v^2_j \\
                     v^1_j & v^2_j & v^1_j + v^2_j \\
                   \end{array}
                 \right]
\]
for $h \in \ZZ$. Further,
\[
d\F(\{\omega_j\}) = \left[
                   \begin{array}{ccc}
                     v^1_j & 0 & v^1_j \\
                     0 & v^2_j & v^2_j \\
                     v^1_j & v^2_j & v^1_j + v^2_j \\
                   \end{array}
                 \right],
\quad j=1, \dots, n .
\]

\subsubsection{A type (1) non-regular process}

We are going to define the first component of $\{\X_t\}$ as a regular, white noise process, the second component as a singular process. We define $\{\X_t\}$ by its spectral density \eqref{eq:exmpl_f},
where
\[
f^{11}(\omega) = \frac{1}{2 \pi}, \quad \omega \in [-\pi, \pi], \quad f^{22}(\omega)
= \left\{\begin{array}{ll}
         \frac12 & |\omega| \le 1, \\
         0 & 1 < |\omega| \le \pi.
         \end{array}
  \right.
\]
The rank $r(\omega)$ of $\f(\omega)$ is $2$ if $|\omega| \le 1$ and is $1$ if $1 < |\omega| \le \pi$, so the rank of the process is not constant.

The covariance matrix function $\C(h)$ has the form \eqref{eq:exmpl_C} with
\[
c_{11}(h) = \left\{\begin{array}{ll}
                     1 & h = 0, \\
                     0 & h \in \ZZ \setminus \{0\},
                   \end{array}
            \right.
\quad c_{22}(h) = \left\{\begin{array}{ll}
                           1 & h = 0, \\
                           \frac{\sin h}{h} & h \in \ZZ \setminus \{0\} . \\
                         \end{array}
                  \right.
\]

\subsubsection{A type (2) non-regular process}

Again, we define the process by its spectral density $\f(\omega)$, which has the form \eqref{eq:exmpl_f}, where
\[
f^{11}(\omega) = \frac{1}{2 \pi}, \quad \omega \in [-\pi, \pi], \quad f^{22}(\omega)
= \left\{\begin{array}{ll}
         0 & \omega = 0, \\
         e^{-\frac{1}{|\omega|}} & \omega \in [-\pi, \pi] \setminus \{0\}.
         \end{array}
  \right.
\]
Then $\f$ has continuous and a.e.\ positive entries on $[-\pi, \pi]$, $\rk(\f(\omega)) = 2$ a.e., and by \eqref{eq:det_Lam},
\[
\det \Lam_2(\omega) = \frac{3}{2 \pi} e^{-\frac{1}{|\omega|}} \quad \text{if} \quad \omega \ne 0.
\]
It implies that
\[
\int_{-\pi}^{\pi} \log \det \Lam_2(\omega) d\omega = \int_{-\pi}^{\pi} \left\{\log \left(\frac{3}{2 \pi}\right) - \frac{1}{|\omega|} \right\} d\omega = -\infty .
\]

\subsubsection{The problem of type (3) non-regular processes}

The next example illustrates the difficulty of showing that a non-full rank time series is non-regular of type (3). The difficulty stems from the fact that one has to show that there exists no factorization $\f(\omega) = \tilde{\UU}(\omega) \Lam_2(\omega) \tilde{\UU}^*(\omega)$ with a sub-orthogonal matrix function $\tilde{\UU}(\omega)$, $\omega \in (-\pi, \pi]$,  belonging to the Hardy space $H^{\infty}$.

Consider the time series defined by a parsimonious spectral decomposition \eqref{eq:spectr_decomp_f_short1} of its spectral density $\f$:
\[
\lambda_{1,2}(\omega) = 1, \quad \lambda_3(\omega) = 0, \quad \Lam_2(\omega) = \I_2,
\]
and
\begin{equation}\label{eq:type3}
\tilde{\UU}(\omega) = \left[
                         \begin{array}{cc}
                           g(\omega) & 0 \\
                           0 & \sqrt{2} g(\omega) \\
                           g(\omega) & 0 \\
                         \end{array}
                       \right],
\quad g(\omega) = \frac{e^{-i \omega} + e^{i 2 \omega}}{2 \sqrt{2} \left| \cos \frac32 \omega \right|}, \quad \omega \in [-\pi, \pi] .
\end{equation}
Then
\begin{equation}\label{eq:exmpl223}
\tilde{\UU}^*(\omega) \tilde{\UU}(\omega) = \I_2, \quad \f(\omega) = \tilde{\UU}(\omega) \Lam_2(\omega) \tilde{\UU}^*(\omega) = \left[
                                       \begin{array}{ccc}
                                         \frac12 & 0 & \frac12 \\
                                         0 & 1 & 0 \\
                                         \frac12 & 0 & \frac12 \\
                                       \end{array}
                                     \right] ,
\end{equation}
which means that it defines a process $\X_t = [X^1_t, X^2_t, X^3_t]^T$ such that $X^1_s \perp X^2_t$ and $X^3_t = X^1_t$ for any $s,t \in \ZZ$. This process satisfies conditions (1) and (2) of Theorem \ref{th:generic_reg_suf} with constant rank $r=2$, but the above specific choice of $\tilde{\UU}(\omega)$ does not belong to $H^{\infty}$, as shown next.

Computing the Fourier coefficients $\psi_{11}(1)$ and $\psi_{11}(-2)$ of the entry $u_{11}(\omega)$ of $\tilde{\UU}(\omega)$, we are going to see that both of them are nonzero, so the Fourier series of $\tilde{\UU}(\omega)$ is two-sided, $\tilde{\UU}(\omega)$ does not belong to the Hardy space $H^2$. By \eqref{eq:type3},
\begin{align*}
\psi_{11}(1) &:= \frac{1}{2 \pi} \int_{-\pi}^{\pi} g(\omega) e^{i \omega} d\omega = \frac{1}{2 \pi} \int_{-\pi}^{\pi} \frac{1 + e^{i 3 \omega}}{2 \left| \cos \frac32 \omega \right|} d\omega \\
&= \frac{1}{2 \pi} \int_{-\pi}^{\pi} \frac{(1 + \cos 3 \omega) + i (\sin 3 \omega)}{\sqrt{2 (1 + \cos 3 \omega)}} d\omega = \frac{2}{\pi} \ne 0 .
\end{align*}
Similarly,
\begin{align*}
\psi_{11}(-2) &:= \frac{1}{2 \pi} \int_{-\pi}^{\pi} g(\omega) e^{-i 2\omega} d\omega = \frac{1}{2 \pi} \int_{-\pi}^{\pi} \frac{e^{-i 3\omega} + 1}{2 \left| \cos \frac32 \omega \right|} d\omega \\
&= \frac{1}{2 \pi} \int_{-\pi}^{\pi} \frac{(1 + \cos 3 \omega) - i (\sin 3 \omega)}{\sqrt{2 (1 + \cos 3 \omega)}} d\omega = \frac{2}{\pi} \ne 0 .
\end{align*}

Notwithstanding, with different factorization  $\f(\omega) = \tilde{\UU}(\omega) \Lam_2(\omega) \tilde{\UU}^*(\omega)$ with a differently given $\tilde{\UU}(\omega)$ shows that, in fact, the process defined by the above spectral density $\f$ is regular. For instance, a `good' factorization is
\[
\tilde{\UU}(\omega) = \left[
                         \begin{array}{cc}
                          \frac{1}{\sqrt{2}} & 0 \\
                           0 &  1 \\
                           \frac{1}{\sqrt{2}} & 0 \\
                         \end{array}
                       \right],
\quad \omega \in [-\pi, \pi] .
\]
Example \ref{sse:regular} below is another `good' factorization, differing from this only by a factor $e^{-i \omega}$.

\subsubsection{A candidate for a type (3) non-regular process}

Take the following parsimonious spectral decomposition \eqref{eq:spectr_decomp_f_short1} of the spectral density $\f$ of a time series:
\[
\lambda_{1,2}(\omega) = 1, \quad \lambda_3(\omega) = 0, \quad \Lam_2(\omega) = \I_2,
\]
and
\begin{equation}\label{eq:bad_UU}
\tilde{\UU}(\omega) = \left[
                         \begin{array}{cc}
                           1 & 0 \\
                           0 & r(\omega) e^{i \phi(\omega)} \\
                           0 & \rho(\omega) e^{i \psi(\omega)} \\
                         \end{array}
                       \right],
\end{equation}
where $r, \rho : [-\pi, \pi] \to [0, 1]$ and $\phi, \psi : [-\pi, \pi] \to \mathbb{R}$ are continuous functions, $r^2(\omega) + \rho^2(\omega) \equiv 1$. Then $\tilde{\UU}^*(\omega) \tilde{\UU}(\omega) = \I_2$ and
\begin{equation}\label{eq:type3_f}
\f(\omega) = \left[
                   \begin{array}{ccc}
                     1 & 0 & 0 \\
                     0 & r^2(\omega) & r(\omega) \rho(\omega) e^{i (\phi(\omega) - \psi(\omega))} \\
                     0 & r(\omega) \rho(\omega) e^{- i (\phi(\omega) - \psi(\omega))} & \rho^2(\omega) \\
                   \end{array}
                 \right] .
\end{equation}
Concretely, define
\begin{equation}\label{eq:concrete}
r(\omega) e^{i \phi(\omega)} := \frac13 (e^{-i \omega} + e^{i 2 \omega}), \quad
\rho(\omega) e^{i \psi(\omega)} := c(\omega) (2 + e^{i 3 \omega}) ,
\end{equation}
where $c(\omega)$ is a suitable real valued function, see below. It follows that
\begin{align}\label{eq:r2}
r^2(\omega) &= \frac29 + \frac19(e^{i 3 \omega} + e^{-i3 \omega}) = \frac29 (1 + \cos 3 \omega) < 1, \\
\rho^2(\omega) &= c^2(\omega) (5 + 2(e^{i 3 \omega} + e^{-i 3 \omega}) = c^2(\omega) (1 + 4 (1 + \cos 3 \omega)). \label{eq:rho2}
\end{align}
Choosing
\begin{equation}\label{eq:c2}
c^2(\omega) := \frac{1 - r^2(\omega)}{1 + 4 (1 + \cos 3 \omega))} = \frac{1 - \frac29 (1 + \cos 3 \omega)}{1 + 4 (1 + \cos 3 \omega)}
\end{equation}
satisfies the condition $r^2(\omega) + \rho^2(\omega) \equiv 1$. Then $\rk (\f(\omega)) = 2$ for a.e.\ $\omega \in [-\pi, \pi]$, so the process is of constant rank $2$.

Let us try to factor the density matrix function as $\f(\omega) = \tilde{\UU}(\omega) \Lam_2(\omega) \tilde{\UU}^*(\omega)$ so that each entry of $\tilde{\UU}(\omega)$ belongs to the Hardy space $H^{\infty}$, instead of \eqref{eq:bad_UU} which does not satisfy this condition. It means that the Fourier series of each entry $u_{jk}(\omega)$ should be a finite or infinite complex linear combination of terms $e^{-i \ell \omega}$ with $\ell \ge 0$.

Since $f_{22}(\omega) = u_{22}(\omega) \overline{u}_{22}(\omega)$, by \eqref{eq:type3_f} and \eqref{eq:r2} we must have $u_{22}(\omega) = d_1 + d_2 e^{-i 3 \omega}$ with suitable complex coefficients $d_1, d_2$. Similarly, $f_{33}(\omega) = u_{32}(\omega) \overline{u}_{32}(\omega)$, so by \eqref{eq:type3_f}, \eqref{eq:rho2}, and \eqref{eq:c2} we should have $u_{32}(\omega) = h_1 + h_2 e^{-i 3 \omega}$ with suitable complex coefficients $h_1, h_2$. Then
\[
f_{23}(\omega) = u_{22}(\omega) \overline{u}_{32}(\omega) = d_1 \overline{h}_1 + d_2 \overline{h}_2 + d_2 \overline{h}_1 e^{-i 3 \omega} + d_1 \overline{h}_2 e^{i 3 \omega}.
\]
On the other hand, by \eqref{eq:type3_f}, \eqref{eq:concrete}, and \eqref{eq:c2} it follows that
\[
f_{23}(\omega) = \frac13 (e^{-i \omega} + e^{i 2 \omega}) \left(\frac{1 - \frac29 (1 + \cos 3 \omega)}{1 + 4 (1 + \cos 3 \omega)}\right)^{\frac12} (2 + e^{-i 3 \omega}) .
\]
This is a contradiction, which shows that the spectral density matrix function $\f$ cannot be factored with factor $\tilde{\UU}(\omega) \in H^{\infty}$. So the process does not satisfy condition (3) of Theorem \ref{th:generic_reg_suf}. Unfortunately, since the necessity of this condition for regularity seems unknown, this does not prove that the process is in fact non-regular.

\subsubsection{A non-full rank regular process}\label{sse:regular}

As above, the process is defined by a parsimonious spectral decomposition \eqref{eq:spectr_decomp_f_short1} of its spectral density $\f$:
\[
\lambda_{1,2}(\omega) = 1, \quad \lambda_3(\omega) = 0, \quad \Lam_2(\omega) = \I_2,
\]
and
\[
\tilde{\UU}(\omega) = \left[
                         \begin{array}{cc}
                          \frac{1}{\sqrt{2}} e^{-i \omega} & 0 \\
                           0 &  e^{-i \omega} \\
                           \frac{1}{\sqrt{2}} e^{-i \omega} & 0 \\
                         \end{array}
                       \right],
\quad \omega \in [-\pi, \pi] .
\]
Then
\[
\tilde{\UU}^*(\omega) \tilde{\UU}(\omega) = \I_2, \quad \f(\omega) = \tilde{\UU}(\omega) \Lam_2(\omega) \tilde{\UU}^*(\omega) = \left[
                                       \begin{array}{ccc}
                                         \frac12 & 0 & \frac12 \\
                                         0 & 1 & 0 \\
                                         \frac12 & 0 & \frac12 \\
                                       \end{array}
                                     \right] ,
\]
which means that it defines a process $\X_t = [X^1_t, X^2_t, X^3_t]^T$ such that $X^1_s \perp X^2_t$ and $X^3_t = X^1_t$ for any $s,t \in \ZZ$. This process obviously satisfies all conditions of Theorem \ref{th:generic_reg_suf} with constant rank $r=2$. In fact, $\{\X_t\}$ is a white noise process with covariance matrix function
\[
\C(h) = \delta_{h 0} \Siga = \delta_{h 0} \left[
                                       \begin{array}{ccc}
                                         \pi & 0 & \pi \\
                                         0 & 2 \pi & 0 \\
                                         \pi & 0 & \pi \\
                                       \end{array}
                                     \right] .
\]


\section{Low rank approximation}\label{se:low_rank}

The aim of this section is to approximate a time series of constant rank $r$ with one of smaller rank $k$. This problem was treated by Brillinger (1969) and Brillinger (1981, Chapter 9), where it was called Principal Component Analysis (PCA) in the Frequency Domain. We show the important fact that when the process is regular, the low rank approximation can also be chosen regular.

\subsection{Approximation of time series of constant rank} \label{sse:low_constant_rank}

Assume that $\{\X_t\}$ is a $d$-dimensional stationary time series of constant rank $r$, $1 \le r \le d$. By Rozanov (1967, Section I.9), it is equivalent to the assumption that $\{\X_t\}$ can be written as a sliding summation, that is, as a two-sided infinite moving average. The  spectral density $\f$ of the process has rank $r$ a.e., and so we may write its eigenvalues as
\begin{equation}\label{eq:eigenvalues}
\lambda_1(\omega) \ge \cdots \ge \lambda_r(\omega) > 0 , \quad \lambda_{r+1}(\omega) =  \cdots = \lambda_d(\omega) = 0.
\end{equation}
Also, the parsimonious spectral decomposition of $\f$ is
\begin{equation}\label{eq:spectr_of_spectr}
\f(\omega) = \sum_{j=1}^r \lambda_j(\omega) \uu_j(\omega) \uu^*_j(\omega) = \tilde{\UU}_r(\omega) \tilde{\Lam}_r(\omega) \tilde{\UU}_r^*(\omega) , \qquad \text{a.e.} \quad \omega \in [-\pi, \pi],
\end{equation}
where $\tilde{\Lam}_r(\omega) := \diag[\lambda_1(\omega) , \dots , \lambda_r(\omega)]$, $\uu_j(\omega) \in \CC^d$ $(j=1, \dots , r)$ are the corresponding orthonormal eigenvectors, and $\tilde{\UU}_r(\omega) \in \CC^{d \times r}$ is the matrix of these column vectors.

Now the problem we are treating can be described as follows. Given an integer $k$, $1 \le k \le r$, find a process $\{\X^{(k)}_t\}$ of constant rank $k$ which is a linear transform of $\{\X_t\}$ and which minimizes the distance
\begin{align}\label{eq:min_distance}
\| \X_t - \X^{(k)}_t \|^2 &= \E \left\{(\X_t - \X^{(k)}_t)^* \, (\X_t - \X^{(k)}_t) \right\} \nonumber \\
&= \tr \; \cov \left\{(\X_t - \X^{(k)}_t), (\X_t - \X^{(k)}_t) \right\}, \quad \forall t \in \ZZ.
\end{align}
In Brillinger (1981) this is called \emph{Principal Component Analysis (PCA) in the Frequency Domain}.

Consider Cram\'er's spectral representations of $\{\X_t\}$ and $\{\X^{(k)}_t\}$:
\[
\X_t = \int_{-\pi}^{\pi} e^{i t \omega} d \Z_{\omega}, \quad \X^{(k)}_t = \int_{-\pi}^{\pi} e^{i t \omega} \T(\omega) d \Z_{\omega}, \quad t \in \ZZ,
\]
where $\T(\omega)$ denotes the linear filter that results $\{\X^{(k)}_t\}$.
Then by \eqref{eq:spectr_of_spectr} we can rewrite \eqref{eq:min_distance} as
\begin{align}\label{eq:min_quadr_form}
&\| \X_t - \X^{(k)}_t \|^2 = \tr \int_{-\pi}^{\pi} (\I_d - \T(\omega)) \f(\omega) (\I_d - \T^*(\omega)) \, d \omega \nonumber \\
&= \tr \int_{-\pi}^{\pi} (\I_d - \T(\omega)) \tilde{\UU}_r(\omega) \tilde{\Lam}_r(\omega) \tilde{\UU}_r^*(\omega) (\I_d - \T^*(\omega)) \, d \omega ,
\end{align}
which clearly does not depend on $t \in \ZZ$.

To find the minimizing linear transformation $\T(\omega)$, from now on we fix an $\omega \in [-\pi, \pi]$ and we study the non-negative definite quadratic form
\[
h(\vv, \omega) := \vv^* \f(\omega) \vv = \vv^* \tilde{\UU}_r(\omega) \tilde{\Lam}_r(\omega) \tilde{\UU}_r^*(\omega) \vv ,
\]
as a function of $\vv \in \CC^d$.

Under the assumption \eqref{eq:eigenvalues}, we can find an orthonormal basis of eigenvectors of $\f(\omega)$: $\{\uu_1(\omega), \dots, \uu_d(\omega)\}$ in $\CC^d$. Then any $\vv \in \CC^d$ can be written as a linear combination $\vv = \sum_{j=1}^d c_j(\omega) \uu_j(\omega)$. It is enough to consider the value of the quadratic form $h(\vv, \omega)$ for $\vv = \uu_j(\omega)$, $j=1, \dots, d$:
\[
h(\uu_j(\omega), \omega) = \uu^*_j(\omega) \tilde{\UU}_r(\omega) \tilde{\Lam}_r(\omega) \tilde{\UU}_r^*(\omega) \uu_j(\omega) = \lambda_j(\omega) ,  \quad j=1, \dots, d .
\]
By \eqref{eq:eigenvalues}, we have
\begin{equation}\label{eq:monoton}
h(\uu_1(\omega), \omega) \ge \cdots \ge h(\uu_r(\omega), \omega) > 0, \quad h(\uu_j(\omega), \omega) = 0 \quad \text{if} \quad j = r+1, \dots, d .
\end{equation}
Since each $\uu_j(\omega) \uu^*_j(\omega)$ is a projection of $\CC^d$ onto its subspace spanned by $\uu_j(\omega)$, $\T(\omega)$ must have rank $k$ a.e., and $\I_d = \sum_{j=1}^d \uu_j(\omega) \uu^*_j(\omega)$, it follows by \eqref{eq:min_quadr_form} and \eqref{eq:monoton} that the minimizing linear transformation is the orthogonal projection
\begin{equation}\label{eq:min_lin_transf}
\T(\omega) = \tilde{\UU}_k(\omega) \tilde{\UU}^*_k(\omega) := \sum_{j=1}^k \uu_j(\omega) \uu^*_j(\omega) .
\end{equation}

Thus we have proved that
\begin{equation}\label{eq:Xt_filtered_res}
\X_t^{(k)} = \int_{-\pi}^{\pi} e^{i t \omega} \tilde{\UU}_k(\omega) \tilde{\UU}^*_k(\omega) d\Z_{\omega}, \quad t \in \ZZ.
\end{equation}
Then the spectral density of $\{\X^{(k)}_t\}$ is
\begin{align}\label{eq:density_r}
\f_k(\omega) &= \tilde{\UU}_k(\omega) \tilde{\UU}^*_k(\omega) \tilde{\UU}_r(\omega) \tilde{\Lam}_r(\omega) \tilde{\UU}^*_r(\omega) \tilde{\UU}_k(\omega) \tilde{\UU}_k^*(\omega) \nonumber \\
&= \tilde{\UU}_k(\omega)
\left[
\begin{array}{cc}
I_k & 0_{k \times (r-k)}
\end{array}
\right]
\tilde{\Lam}_r(\omega)
\left[
\begin{array}{c}
I_k \\
0_{(r-k) \times k} \\
\end{array}
\right]
\tilde{\UU}_k^*(\omega) \nonumber \\
&= \tilde{\UU}_k(\omega)
\tilde{\Lam}_k(\omega)
\tilde{\UU}_k^*(\omega), \quad \omega \in [-\pi, \pi] .
\end{align}
Further, the covariance function of $\{\X^{(k)}_t\}$ is
\begin{equation}\label{eq:cov_r}
\C_k(h) := \int_{-\pi}^{\pi} e^{i h \omega} \f_k(\omega) d\omega, \qquad h \in \ZZ .
\end{equation}

The next theorem summarizes the results above.
\begin{thm}\label{th:rank_reduction}
Assume that $\{\X_t\}$ is a $d$-dimensional stationary time series of constant rank $r$, $1 \le r \le d$, with spectral density $\f$. Let \eqref{eq:eigenvalues} and \eqref{eq:spectr_of_spectr} be the spectral decomposition of $\f$.
\begin{itemize}
\item[(a)] Then
\[
\X_t^{(k)} = \int_{-\pi}^{\pi} e^{i t \omega} \tilde{\UU}_k(\omega) \tilde{\UU}^*_k(\omega) d\Z_{\omega}, \quad t \in \ZZ,
\]
is the approximating process of rank $k$, $1 \le k \le r$, which minimizes the mean square error of the approximation.

\item[(b)] For the mean square error we have
\begin{equation}\label{eq:Xt_Xrt_error}
\| \X_t - \X_t^{(k)} \|^2 = \int_{-\pi}^{\pi} \sum_{j=k+1}^r \lambda_j(\omega) \; d\omega, \qquad t \in \ZZ,
\end{equation}
and
\begin{equation}\label{eq:rel_error1}
\frac{\| \X_t - \X_t^{(k)} \|^2}{\| \X_t \|^2} = \frac{\int_{-\pi}^{\pi} \sum_{j=k+1}^r \lambda_j(\omega) \; d\omega}{\int_{-\pi}^{\pi} \sum_{j=1}^r \lambda_j(\omega) \; d\omega} , \quad t \in \ZZ .
\end{equation}

\item[(c)] If condition
\begin{equation}\label{eq:eps_bound}
\lambda_k(\omega) \ge \Delta > \epsilon \ge \lambda_{k+1}(\omega) \qquad \forall \omega \in [-\pi, \pi],
\end{equation}
holds then we also have
\[
\| \X_t - \X_t^{(k)} \| \le (2 \pi (r-k) \epsilon)^{1/2} , \quad t \in \ZZ
\]
and
\[
\frac{\| \X_t - \X_t^{(k)} \|}{\| \X_t \|} \le \left(\frac{(r-k) \epsilon}{k \Delta}\right)^{\frac12}, \qquad t \in \ZZ .
\]
\end{itemize}
\end{thm}
\begin{proof}
Statement (a) was shown above.

By \eqref{eq:min_quadr_form} and \eqref{eq:min_lin_transf},
\begin{align*}
  &\| \X_t - \X_t^{(k)} \|^2 \\
  &= \tr \int_{-\pi}^{\pi} \left\{\sum_{j=k+1}^{r} \uu_j(\omega) \uu^*_j(\omega) \right\} \tilde{\UU}_r(\omega) \tilde{\Lam}_r(\omega) \tilde{\UU}_r^*(\omega) \left\{\sum_{j=k+1}^{r} \uu_j(\omega) \uu^*_j(\omega)\right\} d\omega \\
  &= \tr \int_{-\pi}^{\pi} \left[
                             \begin{array}{cccc}
                               0_{d \times k} & \uu_{k+1}(\omega) & \cdots & \uu_r(\omega)
                             \end{array}
                           \right]
  \tilde{\Lam}_r(\omega) \left[
                           \begin{array}{c}
                             0_{k \times d} \\
                             \uu^*_{k+1}(\omega) \\
                             \vdots \\
                             \uu^*_r(\omega)
                           \end{array}
                         \right] d\omega \\
  &= \tr \int_{-\pi}^{\pi} \tilde{\UU}_r(\omega) \diag[0, \dots , 0, \lambda_{k+1}, \dots , \lambda_r] \tilde{\UU}^*_r(\omega) d\omega \\
  &= \int_{-\pi}^{\pi} \sum_{j=k+1}^r \lambda_j(\omega) \; d\omega ,
\end{align*}
where we finally used that the trace equals the sum of the eigenvalues of a matrix. This proves \eqref{eq:Xt_Xrt_error}. Since \eqref{eq:Xt_Xrt_error} holds for $k=0$ as well, we get \eqref{eq:rel_error1}.

Finally, condition \eqref{eq:eps_bound} and (b) imply (c).
\end{proof}

\begin{cor}\label{co:cov_appr_err}
For the difference of the covariance functions of $\{\X_t\}$ and $\{\X_t^{(k)}\}$ we have the following estimate:
\[
\| \C(h) - \C_k(h) \| \le \int_{-\pi}^{\pi} \lambda_{k+1}(\omega) d\omega, \quad h \in \ZZ,
\]
where the matrix norm is the spectral norm. If condition \eqref{eq:eps_bound} holds then we have the bound
\[
\| \C(h) - \C_k(h) \| \le 2 \pi \epsilon , \quad h \in \ZZ .
\]
\end{cor}
\begin{proof}
By \eqref{eq:spectr_of_spectr}, \eqref{eq:density_r} and \eqref{eq:cov_r} it follows that
\begin{align*}
&\| \C(h) - \C_k(h) \| = \left\| \int_{-\pi}^{\pi} e^{i h \omega} (f(\omega) - f_k(\omega)) d\omega \right\| \\
&= \left\| \int_{-\pi}^{\pi} e^{i h \omega} \tilde{\UU}_r(\omega) \left\{\tilde{\Lam}_r(\omega) - \diag[\lambda_1(\omega), \dots , \lambda_k(\omega), 0, \dots , 0]\right\} \tilde{\UU}^*_r(\omega) d\omega \right\| \\
&\le \int_{-\pi}^{\pi} \|\tilde{\UU}_r(\omega)\| \cdot \left\| \diag[0, \dots , 0, \lambda_{k+1}(\omega), \dots , \lambda_r(\omega)] \right\| \cdot \|\tilde{\UU}^*_r(\omega)\| d\omega \\
&= \int_{-\pi}^{\pi} \left\| \diag[0, \dots , 0, \lambda_{k+1}(\omega), \dots , \lambda_r(\omega)] \right\| d\omega \\
&= \int_{-\pi}^{\pi} \lambda_{k+1}(\omega) d\omega .
\end{align*}
\end{proof}

Equation \eqref{eq:Xt_filtered_res} can be factored. One can take the Fourier series of the sub-unitary matrix function $\tilde{\UU}_k(\omega) \in L^2$:
\[
\tilde{\UU}_k(\omega) = \sum_{j=-\infty}^{\infty} \psiv(j) e^{-i j \omega}, \quad \psiv(j) = \frac{1}{2 \pi} \int_{-\pi}^{\pi} e^{i j \omega} \tilde{\UU}_k(\omega) d\omega \in \CC^{d \times k},
\]
where $\sum_{j=-\infty}^{\infty} \| \psiv(j) \|^2_F < \infty$. Consequently,
\[
\tilde{\UU}^*_k(\omega) = \sum_{j=-\infty}^{\infty} \psiv^*(j) e^{i j \omega} , \quad \omega \in [-\pi, \pi] .
\]
If the time series $\{\V_t\}$ is defined by the linear filter \eqref{eq:Vt_filter_def}, then similarly to \eqref{eq:Vt_MA} it follows that $\{\V_t\}$ can be obtained from the original time series $\{\X_t\}$ by a sliding summation:
\[
\V_t = \sum_{j=-\infty}^{\infty} \psiv^*(j) \X_{t+j}, \quad t \in \ZZ,
\]
and similarly to \eqref{eq:Vtprocess}, its spectral density is a diagonal matrix:
\[
f_{\V}(\omega) = \Lada_k(\omega) = \diag[\lambda_1(\omega), \dots , \lambda_k(\omega)] .
\]
It means that the covariance matrix function of $\{\V_t\}$ is also diagonal:
\[
\C_{\V}(h) = \diag[c_{11}(h), \dots , c_{kk}(h)], \quad c_{jj}(h) = \int_{-\pi}^{\pi} e^{i h \omega} \lambda_j(\omega) d\omega , \quad h \in \ZZ,
\]
that is, the components of the process $\{\V_t\}$ are orthogonal to each other. Thus the process $\{\V_t\}$ can be called $k$-dimensional \emph{Dynamic Principal Components (DPC)} of the $d$-dimensional process $\{\X_t\}$.

Using a second linear filtration, which is the adjoint $\psiv$ of the previous filtration $\psiv^*$, one can obtain the $k$-rank approximation $\{\X^{(k)}_t\}$ from $\{\V_t\}$:
\begin{align}\label{eq:k_dim_reconstruct}
\X^{(k)}_t &= \int_{-\pi}^{\pi} e^{i t \omega} \tilde{\UU}_k(\omega) \tilde{\UU}^*_k(\omega) d\Z_{\omega} = \int_{-\pi}^{\pi} e^{i t \omega} \tilde{\UU}_k(\omega) d\Z^{\V}_{\omega} \nonumber \\
&= \sum_{j=-\infty}^{\infty} \psiv(j) \V_{t-j} , \quad t \in \ZZ.
\end{align}

Notice the \emph{dimension reduction} in this approximation. Dimension $d$ of the original process $\{\X_t\}$ can be reduced to dimension $k < d$ with the cross-sectionally orthogonal DPC process $\{\V_t\}$, obtained by linear filtration, from which the low-rank approximation $\{\X^{(k)}_t\}$ can be reconstructed also by linear filtration. Of course, this is useful only if the error of the approximation given by Theorem \ref{th:rank_reduction} is small enough.

Since $\tilde{\UU}_k \tilde{\UU}^*_k \in L^2$ as well, one can take the $L^2$-convergent Fourier series
\[
\tilde{\UU}_k(\omega) \tilde{\UU}^*_k(\omega) = \sum_{j, \ell = -\infty}^{\infty} \psiv(j) e^{-i j \omega} \psiv^*(\ell) e^{i \ell \omega} = \sum_{m=-\infty}^{\infty} \wv(m) e^{- i m \omega} ,
\]
where $\omega \in [-\pi, \pi]$ and
\[
\wv(m) = \sum_{j = -\infty}^{\infty} \psiv(j) \, \psiv^*(j-m) \in \CC^{d \times d} , \quad \sum_{m=-\infty}^{\infty} \| \wv(m) \|^2_F < \infty .
\]
It implies that the filtered process $\{\X_t^{(k)}\}$ can be obtained directly from $\{\X_t\}$ by a two-sided sliding summation:
\[
\X_t^{(k)} = \sum_{m=-\infty}^{\infty} \wv(m) \X_{t-m} , \quad t \in \ZZ .
\]

\begin{rem}\label{re:dyn_factor}
By Theorem \ref{th:rank_reduction}, a $d$-dimensional time series $\{\X_t\}$ can be approximated by a rank $k < d$ time series $\{\X^{(k)}_t\}$ with error process $\{\Y^{(k)}_t\}$:
\[
\X_t = \X^{(k)}_t + \Y^{(k)}_t, \qquad t \in \ZZ,
\]
and if for the smallest $d-k$ eigenvalues of the spectral density $\f$ of $\{\X_t\}$ we have a uniform upper bound $\lambda_j(\omega) \le \epsilon$ for any $j=k+1, \dots , d$ and for any $\omega \in [-\pi, \pi]$, then we have a uniform mean square error bound for the error process:  $\|\Y^{(k)}_t\|^2 \le 2 \pi (d-k) \epsilon$. Behind the rank $k$ process $\{\X^{(k)}_t\}$ there is a $k$-dimensional DPC process $\{\V_t\}$ with orthogonal components, from which one can reconstruct $\{\X^{(k)}_t\}$ by \eqref{eq:k_dim_reconstruct}.

\end{rem}

\subsection{Approximation of regular time series}\label{sse:low_regular}

\begin{pro}\label{pr:appr_regular}
Assume that $\{\X_t\}$ is a $d$-dimensional time series of rank $r$ which satisfies all the three conditions of Theorem \ref{th:generic_reg_suf}. Then $\{\X_t\}$ is regular and the rank $k$ approximation $\{\X_t^{(k)}\}$, $1 \le k \le r$, defined in \eqref{eq:Xt_filtered_res} is also a regular time series.
\end{pro}
\begin{proof}
We want to check that the conditions in Theorem \ref{th:generic_reg_suf} hold for $\{\X_t^{(k)}\}$. By \eqref{eq:density_r}, $\{\X_t^{(k)}\}$ has an absolutely continuous spectral measure with density of constant rank $k$, so condition (1) holds.

If $\int_{-\pi}^{\pi} \log \lambda_k(\omega) d\omega = -\infty$, then that would contradict the regularity of the original process $\{\X_t\}$, and this proves condition (2).

It follows from Theorem \ref{th:generic_reg_suf} that $\tilde{\UU}_r(\omega)$ belongs to the Hardy space $H^{\infty}$, so the same holds for $\tilde{\UU}_k(\omega)$ as well, since its columns are subset of the former's. This proves condition (3).
\end{proof}

Theorem \ref{th:rank_reduction} and Corollary \ref{co:cov_appr_err} are valid for regular processes without change. However, the factorization of the approximation discussed above is different in the regular case, because several of the summations become one-sided. Thus we have
\[
\tilde{\UU}^*_k(\omega) = \sum_{j=0}^{\infty} \psiv^*(j) e^{i j \omega} , \quad \omega \in [-\pi, \pi] .
\]
Consequently, the $k$-dimensional, cross-sectionally orthogonal process
$\{\V_t\}$ becomes
\[
\V_t = \sum_{j=0}^{\infty} \psiv^*(j) \X_{t+j}, \quad t \in \ZZ.
\]
Further, the reconstruction of the $k$-rank approximation $\{\X^{(k)}_t\}$ from $\{\V_t\}$ is
\[
\X^{(k)}_t = \sum_{j=0}^{\infty} \psiv(j) \V_{t-j} , \quad t \in \ZZ.
\]

The direct evaluation of $\{\X^{(k)}_t\}$ from $\{\X_t\}$ takes now the following form:
\[
\tilde{\UU}_k(\omega) \tilde{\UU}^*_k(\omega) = \sum_{j, \ell = 0}^{\infty} \psiv(j) e^{-i j \omega} \psiv^*(\ell) e^{i \ell \omega} = \sum_{m=-\infty}^{\infty} \wv(m) e^{- i m \omega} , \quad \omega \in [-\pi, \pi],
\]
where $\wv(m) = \sum_{j = \max(0, m)}^{\infty} \psiv(j) \, \psiv^*(j-m)$.
It implies that the filtered process $\{\X_t^{(k)}\}$ is \emph{not} causally subordinated to the original regular process $\{\X_t\}$ in general, since it can be obtained from $\{\X_t\}$ by a two-sided sliding summation:
\[
\X_t^{(k)} = \sum_{m=-\infty}^{\infty} \wv(m) \X_{t-m} , \quad t \in \ZZ .
\]
On the other hand, it is clear that if $\| \psiv(j) \|_F$ goes to $0$ fast enough as $j \to \infty$, one does not have to use too many `future' terms of $\{\X_t\}$ to get a good enough approximation of $\{\X_t^{(k)}\}$. In practice one can also replace the future values of $\{\X_t\}$ by $\0$ to get a causal approximation of $\{\X^{(k)}_t\}$.


\section*{Acknowledgement}

The research underlying the present paper and carried out at the Budapest
University of Technology and Economics was supported by the
National Research Development and Innovation Fund based on the charter
of bolster issued by the National Research, Development and Innovation
Office under the auspices of the Hungarian Ministry for Innovation and Technology, so
was supported by the National Research, Development and Innovation Fund
(TUDFO/51757/2019-ITM, Thematic Excellence Program).

The author is obliged to Professor Marianna Bolla for her very useful advice and grateful to the anonymous referees for their helpful suggestions.


\section*{Data Availability Statement}

Data sharing is not applicable to this article as no datasets were generated or analysed during the current study.



\begin{thebibliography}{1}

\bibitem[]{brillinger1969canonical}
Brillinger DR. 1969. The canonical analysis of stationary time series.
  Multivariate Analysis II (Krishnaiah PR, ed.) Academic Press, New York,
  pp.~331--350.

\bibitem[]{brillinger1981time}
Brillinger DR. 1981. Time Series: Data Analysis and Theory. vol.~36, SIAM.

\bibitem[]{brockwell1991time} Brockwell PJ, Davis RA, and Fienberg SE. 1991. Time series: Theory
and Methods. Springer, New York.

\bibitem[]{fuhrmann2014linear} Fuhrmann PA. 2014. Linear Systems and Operators in Hilbert Space. Courier Corporation.

\bibitem[]{kolmogorov1941stationary}
Kolmogorov AN. 1941. Stationary sequences in {Hilbert} space. Moscow University
Mathematics Bulletin (in Russian), 2(6):1--40. Translated
into English: Selected works of A. N. Kolmogorov, Vol. II, Probability Theory
and Mathematical Statistics, ed. A. N. Shiryayev, 228--271, Springer,
1992.

\bibitem[]{nikolski2019hardy} Nikolski N. 2019. Hardy spaces, volume 179. Cambridge University Press.

\bibitem[]{rozanov1967stationary}
Rozanov YA. 1967. Stationary Random Processes. Holden-Day.

\bibitem[]{rudin2006real}
Rudin W. 2006. Real and Complex Analysis. Tata McGraw-Hill Education.

\bibitem[]{wiener1957prediction}
Wiener N and Masani P. 1957. The prediction theory of multivariate stochastic
  processes, I. The regularity condition. Acta Mathematica 98: 111--150.

\end{thebibliography}
\end{document}